\documentclass[a4paper,reqno,11pt]{amsart}

\usepackage[utf8]{inputenc}
\usepackage[T1]{fontenc}
\usepackage{amsmath,amssymb,amsfonts,amsthm}
\usepackage{mathtools}
\usepackage{lmodern}
\usepackage{microtype}
\usepackage[margin=1in]{geometry}
\usepackage{stmaryrd}
\usepackage{graphicx}
\usepackage{float}
\usepackage{url}
\usepackage{enumitem}
\theoremstyle{plain}
\newtheorem{theorem}{Theorem}
\newtheorem{remark}{Remark}
\newtheorem{proposition}{Proposition}

\newtheorem{lemma}{Lemma}

\newtheorem{corollary}{Corollary}
\theoremstyle{definition}
\newtheorem{definition}{Definition}

\numberwithin{equation}{section}

\title[On the Class-Breadth Conjecture for $p$-Groups]{On the Class-Breadth Conjecture for $p$-Groups}

\author[A. A. Skutin]{Alexander A. Skutin}

\address{Lomonosov Moscow State University, Faculty of Mechanics and Mathematics, Department of Higher Algebra, Leninskie Gory 1, Moscow, 119991 Russia}
\address{Moscow Center of Fundamental and Applied Mathematics, Moscow, Russia}
\email{a.skutin@mail.ru}

\thanks{This research was supported by the Russian Science Foundation under grant no. 22-11-00075.}
\date{}

\begin{document}

\begin{abstract}
    The class-breadth conjecture of Leedham-Green, Neumann and Wiegold states that the nilpotency class of any $p$-group is at most $b(G) + 1$, where $\displaystyle{b(G) = \max_{g\in G}\log_p[G:Z_G(g)]}$ denotes the breadth of a group $G$. While several counter-examples to this conjecture have been found for $p = 2$, it is still open in general for $p > 2$. This article is dedicated to the general case $p > 2$ of the conjecture.
\end{abstract}

\maketitle

\section{Introduction}

Let $G$ be a finite $p$-group. The breadth $b(x)$ of an element $x$ of a group $G$ is defined as $b(x) := \log_p|G:Z_G(x)|$, where $Z_G(x)$ denotes the centralizer of $x$ in $G$. The breadth $b(G)$ of $G$ is defined as $\displaystyle{b(G) := \max_{g\in G}b(g)}$.

C.~Leedham-Green, P.~Neumann, and J.~Wiegold in \cite{1} conjectured that $\operatorname{cl}(G)\leq b(G) + 1$ for all $p$-groups $G$.

W.~Felsch, J.~Neubüser, and W.~Plesken in \cite{2} provided a series of counterexamples to the conjecture for $p = 2$. Further counterexamples were found by B.~Eick, M.~Newman, and E.~O'Brien in \cite{3}. In \cite{3}, they also proved that the analogue of the conjecture for finite-dimensional nilpotent Lie algebras is false for all finite fields; i.e., for every finite field $\mathbb{F}$, there exists a finite-dimensional nilpotent Lie algebra $\mathfrak{g}$ over $\mathbb{F}$ such that $\operatorname{cl}(\mathfrak{g}) > b(\mathfrak{g}) + 1$, where $\displaystyle{b(\mathfrak{g}) := \max_{a\in\mathfrak{g}}(\dim \mathfrak{g} - \dim Z_{\mathfrak{g}}(a))}$ and $Z_{\mathfrak{g}}(a) := \{x\in\mathfrak{g} \mid [a, x] = 0\}$. 

The conjecture has been proved in some special cases; see, for example: if $b(G)\leq 4$ in \cite{7, 8}, if $b(G)\leq p + 1$ in \cite{4}, if $G$ is a $p$-group of maximal class in \cite{6}, if $G$ is metabelian in \cite{1}, and if $G$ is not covered by its two-step centralizers in \cite{1}. See also \cite{3} for more classes of $p$-groups satisfying the conjecture. A more detailed description of the results related to the class-breadth conjecture can be found in the introduction of \cite{2}.

This article is dedicated to the general case $p > 2$ of the conjecture.

\subsection*{Notation and conventions}

For any $p$-group $G$, any its subgroups $H, H_1, H_2,\ldots, H_n$, any subset $S\subseteq G$, and any elements $a, b\in G$, we denote:

\begin{itemize}
    \item[--] $[a, b] := aba^{-1}b^{-1}$,
    \item[--] $a^b := bab^{-1}$,
    \item[--] $H^a := aHa^{-1} = \{aha^{-1}\mid h\in H\}$,
    \item[--] $\langle S\rangle$ -- the subgroup of $G$ generated by $S$,
    \item[--] $\prod_{i=1}^nH_i = H_1H_2\ldots H_n = H_1\cdot H_2\cdot\ldots\cdot H_n := \langle \cup_{i=1}^nH_i\rangle$ (it should be noted that this notation differs from the standard one, in particular, $H_1H_2\ldots H_n = H_{\sigma(1)}H_{\sigma(2)}\ldots H_{\sigma(n)}$ for each permutation $\sigma\in S_n$),
    \item[--] for each subgroups $A$, $B$ of $G$, define $[A, B] = [B, A] := \langle\{[a, b]\mid a\in A, b\in B\}\rangle$,
    \item[--] $[a, H] := \{[a, h]\mid h\in H\}$,
    \item[--] $[H, a] := \{[h, a]\mid h\in H\}$,
    \item[--] $\langle\!\langle S\rangle\!\rangle^H := \underbrace{\langle S\rangle [\langle S\rangle , H]}_{\mathclap{= \langle\langle S\rangle\cup [\langle S\rangle , H]\rangle}} = $ the normal closure of $S$ in $\underbrace{\langle S\rangle H}_{\mathclap{= \langle S\cup H\rangle}}$,
    \item[--] $Z_G(a) := \{x\in G \mid [x, a] = e\}$,
    \item[--] $Z_G(H) := \cap_{h\in H}Z_G(h)$,
    \item[--] $b_G(a) := \log_p[G : Z_G(a)] = \log_p|[a, G]|$,
    \item[--] $b(G) := \max_{g\in G}b_G(g)$,
    \item[--] $\gamma_i(G)$ -- the $i$-th term of the lower central series of $G$ ($\gamma_1(G) = G, \gamma_2(G) = G'$),
    \item[--] $\operatorname{cl}(G)$ -- the nilpotency class of $G$,
    \item[--] for each element $h\in H$, denote $b_H(h) := \log_p[H : H\cap Z_G(h)]$,
    \item[--] by a maximal subgroup $H$ of $G$, we mean a proper subgroup not contained in any other proper subgroup of $G$; such a subgroup is always normal in $G$ and has index $p$ in $G$,
    \item[--] we denote that $H$ is a maximal subgroup of $G$ as $H\subset_m G$,
    \item[--] we denote that $H$ is a normal subgroup of $G$ as $H\trianglelefteq G$.
\end{itemize}

\subsection*{Acknowledgments}

The author thanks Y. V. Novikova and the Huawei Moscow Research center. The author also thanks the Theoretical Physics and Mathematics Advancement Foundation “BASIS”.

\section{Main Theorem}\label{section2}

\begin{definition}
    Let $G\ltimes P$ be any semidirect product of any $p$-groups $G$ and $P$. A subgroup $Q$ of $P$ is called \emph{$G$-invariant}, if $Q^g = Q$, $\forall g\in G$. We denote that $Q$ is a $G$-invariant subgroup of $P$ as $Q\subseteq^G P$.
\end{definition}

\begin{remark}
Hereafter, for each semidirect product $G \ltimes P$, we will identify 
the groups $G$ and $P$ with the subgroups $G \ltimes \{e\}$ and 
$\{e\} \ltimes P$ of the group $G \ltimes P$, respectively.
\end{remark}

\begin{definition}
    Let $G\ltimes P$ be any semidirect product of any $p$-groups $G$ and $P$. Consider any $G$-invariant subgroup $Q$ of $P$. Define $G$-invariant subgroups $\gamma_i^G(Q)$ ($i\geq 0$) of $P$ inductively as follows:
    \begin{itemize}
        \item $\gamma_0^G(Q) := Q$,

        \medskip
        
        \item $\gamma_1^G(Q) := [Q, \underbrace{GQ}_{\mathclap{= G\ltimes Q}}] = Q'[G, Q]$,
        
        \medskip
        
        \item $\gamma_i^G(Q) := \gamma_1^G(\gamma_{i-1}^G(Q))$, $\forall i\geq 1$.
    \end{itemize}
\end{definition}

The following proposition is straightforward and we omit its proof.

\begin{proposition}\label{proposition1}
Let $G\ltimes P$ be any semidirect product of any $p$-groups $G$ and $P$. Then
\begin{enumerate}
    \item $\gamma_n^G(\gamma_1^G(Q)) = \gamma_1^G(\gamma_n^G(Q)) = \gamma_{n+1}^G(Q)\subseteq\gamma_n^G(Q)$ for each $G$-invariant subgroup $Q$ of $P$ and all $n\geq 0$,
    \item $\gamma_n^G(Q_1)\subseteq\gamma_n^G(Q_2)$ for any $G$-invariant subgroups $Q_1\subseteq Q_2$ of $P$ and all $n\geq 0$.
\end{enumerate}
\end{proposition}

\begin{definition}\label{definition3}
    Let $G\ltimes P$ be any semidirect product of any $p$-groups $G$ and $P$. Let $Q$ be any $G$-invariant subgroup of $P$. Define:
        \begin{itemize}
            \item $\operatorname{cl}^G(Q) := \min\{k\geq 0\mid \gamma_k^G(Q) = \{e\}\}$,

            \medskip
            
    \item $\mathcal{M}^G(Q) := \min \left\{ k \geq 0 \ \middle|\ 
    \begin{aligned}
        &\text{there exist }p - 1\text{ maximal and } G\text{-invariant subgroups}\\ &Q_1,\ldots, Q_{p - 1} \text{ of } Q 
        \text{ such that }b_{GQ}(x) \leq k, \forall x \in Q\setminus \cup_{i = 1}^{p - 1}Q_i 
    \end{aligned} \right\}$.
        \end{itemize}
\end{definition}

The main theorem of this paper is the following:

\begin{theorem}\label{theorem1}
    Let $p > 2$ be a prime number. Let $G\ltimes P$ be any semidirect product of any $p$-groups $G$ and $P$. Then
    \begin{equation*}
\operatorname{cl}^G(P)\leq 
\begin{cases}
    \mathcal{M}^G(P) + 1 & \text{if } [G, P] = \{e\}, \\
    \mathcal{M}^G(P) + \operatorname{cl}^G([G, P]) & \text{otherwise}.
\end{cases}
\end{equation*}
\end{theorem}

\begin{corollary}\label{corollary1}
    Let $p > 2$ be a prime number. Consider any semidirect product $G\ltimes P$ of any $p$-groups $G$ and $P$. Assume that $[G, [G, P]] = \{e\}$. Then $\operatorname{cl}^G(P)\leq \mathcal{M}^G(P) + 1$.
\end{corollary}

\begin{proof}

From the Hall--Witt Three Subgroups Lemma and \cite[Theorem 5.1.7]{robinson1996}, it follows that
\begin{equation*}\begin{split}
    \gamma_1^G([G, P]) &= \underbrace{[G, [G, P]]}_{\mathclap{= \{e\}}}[G, P]' \\ &\subseteq \underbrace{[\overbrace{[[G, P], G]}^{\mathclap{= \{e\}}}, P]}_{\mathclap{= [\{e\}, P] = \{e\}}}\underbrace{[\overbrace{[\underbrace{[G, P]}_{\mathclap{\trianglelefteq GP}}, P]}^{\mathclap{\subseteq [G, P]}}, G]}_{\mathclap{\subseteq [[G, P], G] = \{e\}}} = \{e\}.\end{split}
\end{equation*}
Therefore, from Theorem \ref{theorem1}, we conclude that
\begin{equation*}
    \operatorname{cl}^G(P)\leq \max(\mathcal{M}^G(P) + 1, \mathcal{M}^G(P) + \underbrace{\operatorname{cl}^G([G, P])}_{\mathclap{\leq 1}})\leq \mathcal{M}^G(P) + 1.
\end{equation*}
\end{proof}

\begin{definition}
    Consider any $p$-group $G$ and any its subgroup $H$. A subgroup $P$ of $G$ is called $H$-invariant, if $P^h = P$, $\forall h\in H$.
\end{definition}

\begin{definition}
    Consider any $p$-group $G$. For each subgroup $H$ of $G$, define:
    \begin{itemize}
        \item $\gamma_0^H(G) := G$,
        
        \medskip
        
        \item $\gamma_i^H(G) := [\gamma_{i-1}^H(G), H\gamma_{i-1}^H(G)]$, $\forall i\geq 1$,

        \medskip
        
        \item $\operatorname{cl}^H(G) := \min\{ k\geq 0\mid \gamma_k^H(G) = \{e\}\}$.
    \end{itemize}
\end{definition}

As a further corollary of Theorem \ref{theorem1}, we obtain the following weakened version of the class-breadth conjecture in the case $p > 2$:

\begin{corollary}
    Let $p > 2$ be a prime number. Consider any $p$-group $G$ and any its subgroup $H$. Consider the trivial semidirect product $\{e\}\ltimes G = \{e\}\times G$. Then
    \begin{equation*}
\operatorname{cl}^H(G)\leq 
\begin{cases}
    \mathcal{M}^{\{e\}}(G) + 1 & \text{if } [G, H] = \{e\}, \\
    \mathcal{M}^{\{e\}}(G) + \operatorname{cl}^H([G, H]) & \text{otherwise}.
\end{cases}
\end{equation*}
\end{corollary}

\begin{proof}
The conjugation action of $H$ on $G$ induces a semidirect product $H^*\ltimes G$, where $H^*\cong H$. From the facts that $\operatorname{cl}^{H^*}(P) = \operatorname{cl}^H(P)$ for each $H$-invariant (and, therefore, $H^*$-invariant) subgroup $P$ of $G$, $\mathcal{M}^{H^*}(G) = \mathcal{M}^{\{e\}}(G)$, and Theorem \ref{theorem1}, we conclude that
\begin{equation*}\begin{split}
    \operatorname{cl}^H(G) &= \operatorname{cl}^{H^*}(G)\leq \\ &\leq \begin{cases}
    \mathcal{M}^{H^*}(G) + 1 = \mathcal{M}^{\{e\}}(G) + 1 & \text{if } \underbrace{[G, H]}_{\mathclap{= [G, H^*]}} = \{e\}, \\
    \underbrace{\mathcal{M}^{H^*}(G)}_{\mathclap{= \mathcal{M}^{\{e\}}(G)}} + \underbrace{\operatorname{cl}^{H^*}([G, H^*])}_{\mathclap{= \operatorname{cl}^H([G, H])}} = \mathcal{M}^{\{e\}}(G) + \operatorname{cl}^H([G, H]) & \text{otherwise}.
\end{cases}\end{split}
\end{equation*}
\end{proof}

\section{Proof of Theorem \ref{theorem1}}

\begin{lemma}\label{lemma1}
    Consider any semidirect product $G\ltimes P$ of any $p$-groups $G$ and $P$. Then
    \begin{equation*}
        \gamma_1^G([G, P])\trianglelefteq GP.
    \end{equation*}
\end{lemma}

\begin{proof}
    From $P\trianglelefteq GP$, $G[G, P] = \langle\!\langle G\rangle\!\rangle^P\trianglelefteq GP$, and the fact that $\gamma_1^G(Q) = (GQ)'\cap P$ for each $G$-invariant subgroup $Q$ of $P$, it follows that
    \begin{equation*}
        \gamma_1^G([G, P]) = \underbrace{(\overbrace{G[G, P]}^{\mathclap{\trianglelefteq GP}})'}_{\mathclap{\trianglelefteq GP}}\cap \underbrace{P}_{\mathclap{\trianglelefteq GP}}\trianglelefteq GP.
    \end{equation*}
\end{proof}

\begin{lemma}\label{lemma2}
    Consider any semidirect product $G\ltimes P$ of any $p$-groups $G$ and $P$. Then for any $G$-invariant subgroups $X_1$, $X_2$ of $[G, P]$ and any normal and $G$-invariant subgroup $N$ of $P$,
    \begin{equation*}
        \langle\!\langle [N, X_1]\rangle\!\rangle^{GX_2}\subseteq [N, X_1]\gamma_1^G([G, P]).
    \end{equation*}
\end{lemma}

\begin{proof}

From \cite[Theorem 5.1.7]{robinson1996}, it follows that $[G, P]\trianglelefteq GP$, therefore,
    \begin{equation*}\begin{split}
        \langle\!\langle [N, X_1]\rangle\!\rangle^{GX_2} &= [N, X_1][\underbrace{[N, X_1]}_{\mathclap{\subseteq [P, [G, P]]\subseteq [G, P]}}, \overbrace{GX_2}^{\mathclap{\subseteq G[G, P]}}]\\ &\subseteq [N, X_1]\underbrace{[[G, P], G[G, P]]}_{\mathclap{= \gamma_1^G([G, P])}} \\ &= [N, X_1]\gamma_1^G([G, P]).\end{split}
    \end{equation*}
\end{proof}

\begin{lemma}\label{lemma3}
    Consider any semidirect product $G\ltimes P$ of any $p$-groups $G$ and $P$. Then for any $G$-invariant subgroups $X_1$, $X_2$ of $[G, P]$ and any normal and $G$-invariant subgroup $N$ of $P$,
    \begin{equation*}
        \gamma_1^G(NX_1X_2) \subseteq \gamma_1^G([G, P])N'[N, GX_1][N, GX_2].
    \end{equation*}
\end{lemma}

\begin{proof}

From Lemma \ref{lemma1}, we obtain $\gamma_1^G([G, P])\trianglelefteq GP$. After the factorization
\begin{equation*}
    \pi : G\ltimes P\to G\ltimes (P/\gamma_1^G([G, P])),
\end{equation*}
we can assume that $\gamma_1^G([G, P]) = \{e\}$. Therefore, from
\begin{itemize}
    \item the Hall--Witt Three Subgroups Lemma,
    \item the fact that $[H, KL]\subseteq [H, K\langle\!\langle L\rangle\!\rangle^K]\subseteq [H, K]\langle\!\langle [H, \langle\!\langle L\rangle\!\rangle^K]\rangle\!\rangle^K$ for all subgroups $H$, $K$, $L$ of $G\ltimes P$,
    \item Lemma \ref{lemma2},
    \item Proposition \ref{proposition1},
    \item the fact that $[G, X_1][G, X_2][X_1, X_2]\subseteq [G, [G, P]][G, P]' = \gamma_1^G([G, P]) = \{e\}$,
\end{itemize}
we conclude that
    \begin{equation*}
        \begin{split}
            \gamma_1^G(NX_1X_2) &= [NX_1X_2, GNX_1X_2] = [GNX_1X_2, X_1X_2N] \\ &\subseteq [GNX_1X_2, X_1X_2]\underbrace{\langle\!\langle [\overbrace{GNX_1X_2}^{\mathclap{\supseteq X_1X_2}}, \overbrace{\langle\!\langle N\rangle\!\rangle^{GP}}^{\mathclap{= N\trianglelefteq GP}}]\rangle\!\rangle^{X_1X_2}}_{\mathclap{= [GNX_1X_2, N]}}\\ &= [X_1X_2, GX_1X_2N][N, GX_1X_2N] \\ &\subseteq \underbrace{[X_1X_2, GX_1X_2]}_{\mathclap{= \gamma_1^G(X_1X_2)}}\underbrace{\langle\!\langle [\underbrace{X_1X_2}_{\mathclap{\trianglelefteq GX_1X_2}}, \overbrace{\langle\!\langle N\rangle\!\rangle^{GP}}^{\mathclap{= N\trianglelefteq GP}}]\rangle\!\rangle^{GX_1X_2}}_{\mathclap{= [X_1X_2, N]\subseteq [N, GX_1X_2]}}[N, GX_1X_2]\overbrace{\langle\!\langle [N, \underbrace{\langle\!\langle N\rangle\!\rangle^{GP}}_{\mathclap{= N}}]\rangle\!\rangle^{GP}}^{\mathclap{= N'}} \\ &= \underbrace{\gamma_1^G(\overbrace{X_1X_2}^{\mathclap{\subseteq [G, P]}})}_{\mathclap{\subseteq \gamma_1^G([G, P]) = \{e\}}}N'[N, GX_2X_1] \\ &\subseteq N'[N, GX_2]\underbrace{\langle\!\langle [N, \overbrace{\langle\!\langle X_1\rangle\!\rangle^{GX_2}}^{\mathclap{= X_1[X_1, GX_2] = X_1}}]\rangle\!\rangle^{GX_2}}_{\mathclap{= \langle\!\langle [N, X_1]\rangle\!\rangle^{GX_2}\subseteq [N, X_1]\gamma_1^G([G, P]) = [N, X_1]}} \\ &\subseteq \gamma_1^G([G, P])N'[N, GX_1][N, GX_2].
        \end{split}
    \end{equation*}
\end{proof}

\begin{lemma}\label{lemma4}
    Consider any semidirect product $G\ltimes P$ of any $p$-groups $G$ and $P$. Then for any $G$-invariant subgroups $X_1$, $X_2$ of $[G, P]$ and any normal and $G$-invariant subgroup $N$ of $P$,
    \begin{equation*}
        \gamma_i^G(NX_1X_2) \subseteq \gamma_1^G([G, P])\gamma_i^G(NX_1)\gamma_i^G(NX_2)\quad \forall i\geq 0.
    \end{equation*}
\end{lemma}

\begin{proof}
    Induction on $i\geq 0$. The base case $i = 0$ is trivial. Proof of the induction step:

    From Lemma \ref{lemma1}, we obtain $\gamma_1^G([G, P])\trianglelefteq GP$. After the factorization
\begin{equation*}
    \pi : G\ltimes P\to G\ltimes (P/\gamma_1^G([G, P])),
\end{equation*}
we can assume that $\gamma_1^G([G, P]) = \{e\}$. From $X_1, X_2\subseteq [G, P]$ and $P\trianglelefteq GP$, we obtain
    \begin{equation*}\begin{split}
        [P, GX_i] &\subseteq [P, \underbrace{G[G, P]}_{\mathclap{= \langle\!\langle G\rangle\!\rangle^P\trianglelefteq GP}}]\cap P \subseteq (\overbrace{G[G, P]}^{\mathclap{= G\ltimes [G, P]}})\cap \underbrace{P}_{\mathclap{= \{e\}\ltimes P}} \\ &= \{e\}\ltimes [G, P] = [G, P]\quad \forall 1\leq i\leq 2.\end{split}
    \end{equation*}
    Therefore, for $i\geq 1$, from
    \begin{itemize}
        \item Proposition \ref{proposition1},
        \item the fact that $\gamma_1^G([G, P]) = \{e\}$,
        \item Lemma \ref{lemma1},
        \item Lemma \ref{lemma3},
        \item the fact that $[N, GX_i]\subseteq [P, GX_i]\subseteq [G, P]$, $\forall 1\leq i\leq 2$,
        \item the induction hypothesis,
    \end{itemize}
    we conclude that
    \begin{equation*}
        \begin{split}
            \gamma_i^G(NX_1X_2) &= \gamma_{i-1}^G(\gamma_1^G(NX_1X_2)) \\ &\subseteq \gamma_{i-1}^G(\underbrace{\gamma_1^G([G, P])}_{\mathclap{= \{e\}}}N'[N, GX_1][N, GX_2]) \\ &= \gamma_{i-1}^G(\overbrace{N'}^{\mathclap{:= N_{new}\trianglelefteq GP}}\underbrace{[N, GX_1]}_{\mathclap{:= X_{1, new}\subseteq^G [G, P]}}\overbrace{[N, GX_2]}^{\mathclap{:= X_{2, new}\subseteq^G [G, P]}}) \\ &\subseteq \gamma_1^G([G, P])\gamma_{i-1}^G(N_{new}X_{1, new})\gamma_{i-1}^G(N_{new}X_{2, new}) \\ &= \gamma_1^G([G, P])\gamma_{i-1}^G(\underbrace{N'[N, GX_1]}_{\mathclap{\subseteq \gamma_1^G(NX_1)}})\gamma_{i-1}^G(\underbrace{N'[N, GX_2]}_{\mathclap{\subseteq \gamma_1^G(NX_2)}}) \\ &\subseteq \gamma_1^G([G, P])\gamma_i^G(NX_1)\gamma_i^G(NX_2).
        \end{split}
    \end{equation*}
\end{proof}

\subsection{Proof of Theorem \ref{theorem1}}

    Induction on $|G\ltimes P|$. The base case $|G\ltimes P| = 1$ is trivial. Proof of the induction step:

    From Lemma \ref{lemma1}, we obtain $\gamma_1^G([G, P])\trianglelefteq GP$. Consider the homomorphism of factorization
    \begin{equation*}
        \pi : G\ltimes P\to G\ltimes (P/\gamma_1^G([G, P])).
    \end{equation*}
    Without loss of generality, we may assume that $\gamma_1^G([G, P]) = \{e\}$, because if $\gamma_1^G([G, P])\not= \{e\}$, then we obtain that $[G, P]\not= \{e\}$ and, from the fact that $|\pi(G\ltimes P)| = |G\ltimes \pi(P)| = |G\ltimes P|/|\gamma_1^G([G, P])| < |G\ltimes P|$ and the induction hypothesis, we conclude that
    
    \begin{equation*}\begin{split}
        \operatorname{cl}^G(P) &\leq \operatorname{cl}^G(\pi(P)) + \overbrace{\operatorname{cl}^G(\underbrace{\gamma_1^G([G, P])}_{\mathclap{\not= \{e\}}})}^{\mathclap{= \operatorname{cl}^G([G, P]) - 1}}\\ &\leq \max(\mathcal{M}^G(\pi(P)) + \underbrace{\operatorname{cl}^G([G, \pi(P)])}_{\mathclap{\leq 1}}, \mathcal{M}^G(\pi(P)) + 1) + \operatorname{cl}^G([G, P]) - 1\\ &= \mathcal{M}^G(P) + \operatorname{cl}^G([G, P]).\end{split}
    \end{equation*}
    
    Consider the case when $\mathcal{M}^G(P) = 0$.
    
    From $\mathcal{M}^G(P) = 0$ and the fact that the $p$-group $P$ cannot be covered by any its $p$ proper subgroups, it follows that $\gamma_1^G(P) = \{e\}$, i.e. $[G, P] = \{e\}$, $\operatorname{cl}^G(P)\leq 1 = \mathcal{M}^G(P) + 1$ and the theorem is proved.

    Consider the case when $\mathcal{M}^G(P) > 0$.
    
    Consider the case when the group $P/\gamma_1^G(P)$ is cyclic.
    
    From \cite[Theorem 5.1.7]{robinson1996}, it follows that $[G, P]\trianglelefteq P$. Consider the homomorphism of factorization
    \begin{equation*}
        \varphi : P\to P/[G, P].
    \end{equation*}
    Then from $\varphi(P)/\varphi(P)' = \varphi(P)/\varphi(P') = P/(P'[G, P]) = P/\gamma_1^G(P)$ and the fact that the group $P/\gamma_1^G(P)$ is cyclic, it follows that the group $\varphi(P)/\varphi(P)'$ is cylic. Therefore, the group $\varphi(P)$ is cyclic and $P'\subseteq \ker\varphi = [G, P]$. From $P'\subseteq [G, P]$ and $\mathcal{M}^G(P)\geq 1$, we conclude that $[G, P]\not= \{e\}$ (because, otherwise $[P, GP] = [P, G]P' = \{e\}$ and $\mathcal{M}^G(P) = 0$) and
    
        \begin{equation*}
            \operatorname{cl}^G(P)\leq \operatorname{cl}^G(\underbrace{\gamma_1^G(P)}_{\mathclap{= P'[G, P] = [G, P]}}) + 1 = \operatorname{cl}^G([G, P]) + 1 \leq \mathcal{M}^G(P) + \operatorname{cl}^G([G, P]).
        \end{equation*}
    
    Consider the case when the group $P/\gamma_1^G(P)$ is not cyclic.
    
    From the definition of $\mathcal{M}^G(P)$, it follows that there exist $p-1\geq 2$ maximal and $G$-invariant subgroups $P_1, P_2,\ldots, P_{p-1}$ of $P$ such that $\{x\in P \mid b_{GP}(x) > \mathcal{M}^G(P)\}\subseteq \cup_{i=1}^{p-1}P_i$. In the case if $P_1 = P_2$, from the fact that $P/\gamma_1^G(P)$ is not cyclic, $P_2$ can be replaced with any maximal subgroup $P_2^*$ of $P$ that contains $\gamma_1^G(P)$ and is different from $P_1$, where the $G$-invariance of $P_2^*$ follows from the fact that $[G, P_2^*]\subseteq [G, P]\subseteq \gamma_1^G(P)\subseteq P_2^*$. Therefore, we may assume that $P_1\neq P_2$. From $P/(P_1\cap P_2)\cong \mathbb{Z}_p^2$ and $p > 2$, it follows that there exist two different maximal subgroups $Q_1$, $Q_2$ of $P$ such that $P_i\not= Q_j$, $\forall 1\leq i\leq p-1$, $\forall 1\leq j\leq 2$ and $P_1\cap P_2 = Q_1\cap Q_2 = Q_1\cap (P_1\cup P_2) = Q_2\cap (P_1\cup P_2)$. From $[GP : GP_i] = [G\ltimes P : G\ltimes P_i] = [P : P_i] = p$, $\forall 1\leq i\leq p-1$, it follows that $(GP)'\subseteq GP_i$, $\forall 1\leq i\leq p-1$, therefore,
    \begin{equation*}
        \gamma_1^G(P) = [P, GP]\subseteq \underbrace{(GP)'}_{\mathclap{\subseteq GP_i}}\cap P\subseteq (\underbrace{GP_i}_{\mathclap{= G\ltimes P_i}})\cap \overbrace{P}^{\mathclap{= \{e\}\ltimes P}} = \{e\}\ltimes P_i = P_i\quad\forall 1\leq i\leq p-1.
    \end{equation*}
    Therefore, $\gamma_1^G(P)\subseteq P_1\cap P_2$ and from $[G, Q_j]\subseteq [G, P]\subseteq \gamma_1^G(P)\subseteq P_1\cap P_2\subseteq Q_j\subseteq P$, $\forall 1\leq j\leq 2$, it follows that $Q_1$, $Q_2$ are maximal and $G$-invariant subgroups of $P$. From the facts that $[P : P_1\cap P_2] = p^2$ and $[P : Q_j] = p$, $P_1\cap P_2\subseteq Q_j$, $\forall 1\leq j\leq 2$, it follows that $[Q_j : P_1\cap P_2] = p$, $\forall 1\leq j\leq 2$ and $P_1\cap P_2$ is a maximal and $G$-invariant subgroup of $Q_j$, $\forall 1\leq j\leq 2$. From
    \begin{itemize}
        \item $Q_j\cap (P_1\cup P_2) = P_1\cap P_2\subset_m Q_j$, $\forall 1\leq j\leq 2$,
        \item $b_{GQ_j}(x)\leq b_{GP}(x)$, $\forall 1\leq j\leq 2$, $\forall x\in GQ_j$,
        \item $P_i\not= Q_j$, $\forall 1\leq i\leq p-1$, $\forall 1\leq j\leq 2$,
        \item $P_i\subset_m P$, $Q_j\subset_m P$, $\forall 1\leq i\leq p-1$, $\forall 1\leq j\leq 2$,
        \item $\{x\in P \mid b_{GP}(x) > \mathcal{M}^G(P)\}\subseteq \cup_{i=1}^{p-1}P_i$,
    \end{itemize}
    it follows that $\forall 1\leq j\leq 2$,
    
    \begin{equation*}
    \begin{split}
        \{x\in Q_j \mid b_{GQ_j}(x) > \mathcal{M}^G(P)\}&\subseteq\underbrace{\{x\in P \mid b_{GP}(x) > \mathcal{M}^G(P)\}}_{\mathclap{\subseteq \cup_{i=1}^{p-1}P_i}}\cap Q_j\\ &\subseteq \cup_{i=1}^{p-1}(P_i\cap Q_j)\\ &= \underbrace{(Q_j\cap (P_1\cup P_2))}_{\mathclap{= P_1\cap P_2}}\bigcup_{2 < i\leq p-1}(Q_j\cap P_i)\\ &= \underbrace{(P_1\cap P_2)}_{\mathclap{\subset_m Q_j}}\bigcup_{2 < i\leq p-1}\underbrace{(Q_j\cap P_i)}_{\mathclap{\subset_m Q_j}},\end{split}
    \end{equation*}
    therefore, $\mathcal{M}^G(Q_j)\leq \mathcal{M}^G(P)$, $\forall 1\leq j\leq 2$.

    Assume that there exists an integer $1\leq n\leq 2$ such that $\mathcal{M}^G(Q_n) < \mathcal{M}^G(P)$.

    Consider the case when $[G, P] = \{e\}$. Then from
    \begin{itemize}
        \item $[G, Q_n]\subseteq [G, P] = \{e\}$,
        \item $\gamma_1^G(P)\subseteq Q_j$, $\forall 1\leq j\leq 2$,
        \item $\mathcal{M}^G(Q_n) < \mathcal{M}^G(P)$,
        \item Proposition \ref{proposition1},
        \item the induction hypothesis,
    \end{itemize}
    we conclude that $[G, P] = \{e\}$ and
    
    \begin{equation*}
        \operatorname{cl}^G(P)\leq \overbrace{\operatorname{cl}^G(\underbrace{\gamma_1^G(P)}_{\mathclap{\subseteq Q_n}})}^{\mathclap{\leq \operatorname{cl}^G(Q_n)\leq \mathcal{M}^G(Q_n) + 1 \leq \mathcal{M}^G(P)}} + 1\leq \mathcal{M}^G(P) + 1.
    \end{equation*}

    Consider the case when $[G, P]\not= \{e\}$. Then from
    \begin{itemize}
        \item $[G, P]\not= \{e\}$ and, therefore, $\operatorname{cl}^G([G, P])\geq 1$,
        \item $\gamma_1^G(P)\subseteq Q_j$, $\forall 1\leq j\leq 2$,
        \item $\mathcal{M}^G(Q_n) < \mathcal{M}^G(P)$,
        \item Proposition \ref{proposition1},
        \item the induction hypothesis,
    \end{itemize}
    we conclude that $[G, P]\not= \{e\}$ and
    
    \begin{equation*}\begin{split}
        \operatorname{cl}^G(P)&\leq \operatorname{cl}^G(\underbrace{\gamma_1^G(P)}_{\mathclap{\subseteq Q_n}}) + 1\\ &\leq \operatorname{cl}^G(Q_n) + 1\\ &\leq \max(\underbrace{\mathcal{M}^G(Q_n)}_{\mathclap{\leq \mathcal{M}^G(P) - 1}} + \overbrace{\operatorname{cl}^G(\underbrace{[Q_n, G]}_{\mathclap{\subseteq [P, G]}})}^{\mathclap{\leq \operatorname{cl}^G([P, G])}}, \underbrace{\mathcal{M}^G(Q_n) + 1}_{\mathclap{\leq \mathcal{M}^G(P)}}) + 1\\ &\leq \max(\mathcal{M}^G(P) - 1 + \underbrace{\operatorname{cl}^G([P, G])}_{\mathclap{\geq 1}}, \mathcal{M}^G(P)) + 1\\ &\leq \mathcal{M}^G(P) + \operatorname{cl}^G([P, G]).\end{split}
    \end{equation*}
    
    Therefore, we can assume that $\mathcal{M}^G(Q_j) = \mathcal{M}^G(P)$, $\forall 1\leq j\leq 2$.
    
    Assume that $[Q_n, GP]\nsubseteq \gamma_1^G(Q_n)$ for some integer $1\leq n\leq 2$.
    
    Consider the set
    \begin{equation*}
        X := \{x \in Q_n \mid [x, GP] \subseteq \gamma_1^G(Q_n)\}.
    \end{equation*}
    From $[GP : GQ_n] = [G\ltimes P : G\ltimes Q_n] = [P : Q_n] = p$, it follows that $GQ_n\trianglelefteq GP$, therefore,
    \begin{equation*}
        \gamma_1^G(Q_n) = \underbrace{(GQ_n)'}_{\mathclap{\trianglelefteq GP}}\cap P\trianglelefteq GP.
    \end{equation*}
    Consider the homomorphism of factorization
    \begin{equation*}
        \psi : GP\to (GP)/\gamma_1^G(Q_n).
    \end{equation*}
    From $Q_n\trianglelefteq GQ_n$ and $X = \psi^{-1}(Z(\psi(GP))\cap \psi(Q_n))\trianglelefteq GQ_n$, it follows that $Q_n\cap X\trianglelefteq GQ_n$.
    
    From $[Q_n, GP]\nsubseteq \gamma_1^G(Q_n)$, we obtain that $X\cap Q_n\subsetneq Q_n$ (because, otherwise, $Q_n\subseteq X$ and $[Q_n, GP]\subseteq \gamma_1^G(Q_n)$). From the facts that
    \begin{itemize}
        \item $[x, GQ_n] \subseteq [Q_n, GQ_n] = \gamma_1^G(Q_n)$, $\forall x\in Q_n$,
        
        \item $GQ_n\subseteq GP$,
        
        \item $b_{GQ_n}(x) = \log_p|[x, GQ_n]|$, $\forall x\in GQ_n$,
        
        \item $b_{GP}(x) = \log_p|[x, GP]|$, $\forall x\in GP$,
    \end{itemize}
    it follows that
    
    \begin{equation*}\begin{split}
        \{x\in Q_n\mid b_{GQ_n}(x) = b_{GP}(x)\} &= \{x\in Q_n\mid [x, GP] = [x, GQ_n]\} \\ &\subseteq \{x\in Q_n\mid [x, GP]\subseteq \gamma_1^G(Q_n)\} = X.\end{split}
    \end{equation*}
    Therefore, from the facts that $\{x\in Q_n\mid b_{GQ_n}(x) = b_{GP}(x)\}\subseteq X$ and $b_{GQ_n}(x)\leq b_{GP}(x)$, $\forall x\in Q_n$, it follows that
    \begin{equation*}
        b_{GQ_n}(x) < b_{GP}(x)\quad \forall x\in Q_n\setminus X.
    \end{equation*}
    Consider any maximal and $G$-invariant subgroup $H$ of $Q_n$ that contains $X\cap Q_n$ (such a group exists from the facts that $X\cap Q_n\subsetneq Q_n$ and $X\cap Q_n, Q_n\trianglelefteq GQ_n$). From the facts that
        \begin{itemize}
            \item $\{x\in P \mid b_{GP}(x) > \mathcal{M}^G(P)\}\subseteq \cup_{i=1}^{p-1}P_i$,
            \item $Q_j\cap (P_1\cap P_2) = Q_j\cap (P_1\cup P_2)$, $\forall 1\leq j\leq 2$ and, therefore, $Q_j\setminus (P_1\cap P_2) = Q_j\setminus (P_1\cup P_2)$, $\forall 1\leq j\leq 2$,
            \item $\mathcal{M}^G(Q_j) = \mathcal{M}^G(P)$, $\forall 1\leq j\leq 2$,
            \item $b_{GQ_n}(x) < b_{GP}(x)$, $\forall x\in Q_n\setminus X$,
            \item $X\cap Q_n\subseteq H$,
        \end{itemize}
        it follows that
        \begin{equation*}
            b_{GQ_n}(x) < b_{GP}(x)\leq \mathcal{M}^G(P) = \mathcal{M}^G(Q_n)\quad \forall x\in Q_n\setminus ((P_1\cap P_2)\cup H\cup_{2 < i \leq p-1}P_i),
        \end{equation*}
        i.e. that the set $\{x\in Q_n\mid b_{GQ_n}(x) > \mathcal{M}^G(Q_n) - 1\}$ can be covered by $p-1$ maximal and $G$-invariant subgroups $P_1\cap P_2$, $H$, $P_i$, $2 < i\leq p-1$ of $Q_n$, which contradicts to the definition of $\mathcal{M}^G(Q_n)$. Therefore, $[Q_n, GP]\subseteq \gamma_1^G(Q_n)$, $\forall 1\leq n\leq 2$.
        
        Consider the case when $[G, P] = \{e\}$. From the induction hypothesis and the facts that $[G, Q_j]\subseteq [G, P] = \{e\}$, $\mathcal{M}^G(Q_j) = \mathcal{M}^G(P)$, $\forall 1\leq j\leq 2$, it follows that $\operatorname{cl}^G(Q_j)\leq \mathcal{M}^G(Q_j) + 1 = \mathcal{M}^G(P) + 1$, $\forall 1\leq j\leq 2$. From the facts that
        \begin{itemize}
            \item $[G, P] = \{e\}$,
            \item the fact that $[H, KL]\subseteq [H, K\langle\!\langle L\rangle\!\rangle^K]\subseteq [H, K]\langle\!\langle [H, \langle\!\langle L\rangle\!\rangle^K]\rangle\!\rangle^K$ for all subgroups $H$, $K$, $L$ of $G\ltimes P$,
            \item the fact that $[Q, P] = P'$ for each maximal subgroup $Q$ of $P$,
            \item $[Q_n, P][Q_n, G]\subseteq [Q_n, GP]\subseteq \gamma_1^G(Q_n)$, $\forall 1\leq n\leq 2$,
            \item $\operatorname{cl}^G(Q_j)\leq \mathcal{M}^G(P) + 1$, $\forall 1\leq j\leq 2$,
            \item Proposition \ref{proposition1},
            \item Lemma \ref{lemma4},
        \end{itemize}
        we conclude that $[G, P] = \{e\}$ and
        
    \begin{equation*}\begin{split}
        \operatorname{cl}^G(P) &\leq \operatorname{cl}^G(\underbrace{\gamma_1^G(P)}_{\mathclap{= [P, GP] = [Q_1Q_2, GP] \subseteq [Q_1, GP]\langle\!\langle [\langle\!\langle Q_2\rangle\!\rangle^{Q_1}, GP]\rangle\!\rangle^{Q_1} = [Q_1, GP][Q_2, GP]}}) + 1 \\ &\leq \operatorname{cl}^G(\underbrace{[Q_1, GP][Q_2, GP]}_{\mathclap{= [Q_1, G]\langle\!\langle [Q_1, \langle\!\langle P\rangle\!\rangle^G]\rangle\!\rangle^G [Q_2, G]\langle\!\langle [Q_2, \langle\!\langle P\rangle\!\rangle^G]\rangle\!\rangle^G = [Q_1, G][Q_1, P][Q_2, G][Q_2, P]}}) + 1 \\ &= \operatorname{cl}^G(\overbrace{[Q_1, P][Q_2, P]}^{\mathclap{= P' := N\trianglelefteq GP}}\underbrace{[Q_1, G]}_{\mathclap{:= X_1\subseteq^G [G, P]}}\overbrace{[Q_2, G]}^{\mathclap{:= X_2\subseteq^G [G, P]}}) + 1 \\ &= \operatorname{cl}^G(NX_1X_2) + 1 \\ &\leq \operatorname{cl}^G(\overbrace{\gamma_{\mathcal{M}^G(P)}^G(NX_1X_2)}^{\mathclap{\subseteq \gamma_1^G([G, P])\gamma_{\mathcal{M}^G(P)}^G(NX_1)\gamma_{\mathcal{M}^G(P)}^G(NX_2)}}) + \mathcal{M}^G(P) + 1 \\ &\leq \operatorname{cl}^G(\underbrace{\gamma_1^G([G, P])}_{\mathclap{= \{e\}}}\gamma_{\mathcal{M}^G(P)}^G(\underbrace{NX_1}_{\mathclap{= P'[Q_1, G]}})\gamma_{\mathcal{M}^G(P)}^G(\underbrace{NX_2}_{\mathclap{= P'[Q_2, G]}})) + \mathcal{M}^G(P) + 1\\ &= \operatorname{cl}^G(\gamma_{\mathcal{M}^G(P)}^G(\underbrace{P'[Q_1, G]}_{\mathclap{= [Q_1, P][Q_1, G]\subseteq\gamma_1^G(Q_1)}})\gamma_{\mathcal{M}^G(P)}^G(\overbrace{P'[Q_2, G]}^{\mathclap{= [Q_2, P][Q_2, G]\subseteq\gamma_1^G(Q_2)}})) + \mathcal{M}^G(P) + 1\\ &\leq \operatorname{cl}^G(\underbrace{\gamma_{\mathcal{M}^G(P) + 1}^G(Q_1)}_{\mathclap{\subseteq \gamma_{\operatorname{cl}^G(Q_1)}^G(Q_1) = \{e\}}}\overbrace{\gamma_{\mathcal{M}^G(P) + 1}^G(Q_2)}^{\mathclap{\subseteq \gamma_{\operatorname{cl}^G(Q_1)}^G(Q_1) = \{e\}}}) + \mathcal{M}^G(P) + 1\\ &= \operatorname{cl}^G(\{e\}) + \mathcal{M}^G(P) + 1 = \mathcal{M}^G(P) + 1.\end{split}
    \end{equation*}
    
    Consider the case when $[G, P] \not= \{e\}$. From the induction hypothesis, Proposition \ref{proposition1}, and the fact that $\mathcal{M}^G(Q_j) = \mathcal{M}^G(P)$, $\forall 1\leq j\leq 2$, it follows that $\forall 1\leq j\leq 2$,
    \begin{equation*}\begin{split}
        \operatorname{cl}^G(Q_j) &\leq \max(\mathcal{M}^G(Q_j) + \operatorname{cl}^G(\underbrace{[G, Q_j]}_{\mathclap{\subseteq [G, P]}}), \mathcal{M}^G(Q_j) + 1)\\ &\leq \max(\mathcal{M}^G(P) + \overbrace{\operatorname{cl}^G(\underbrace{[G, P]}_{\mathclap{\not= \{e\}}})}^{\mathclap{\geq 1}}, \mathcal{M}^G(P) + 1) = \mathcal{M}^G(P) + \operatorname{cl}^G([G, P]).\end{split}\end{equation*}
        
        Denote $k := \mathcal{M}^G(P) + \operatorname{cl}^G([G, P]) - 1$. From the facts that
        \begin{itemize}
            \item $[G, P] \not= \{e\}$,
            \item the fact that $[H, KL]\subseteq [H, K\langle\!\langle L\rangle\!\rangle^K]\subseteq [H, K]\langle\!\langle [H, \langle\!\langle L\rangle\!\rangle^K]\rangle\!\rangle^K$ for all subgroups $H$, $K$, $L$ of $G\ltimes P$,
            \item the fact that $[Q, P] = P'$ for each maximal subgroup $Q$ of $P$,
            \item $[Q_n, P][Q_n, G]\subseteq [Q_n, GP]\subseteq \gamma_1^G(Q_n)$, $\forall 1\leq n\leq 2$,
            \item $\operatorname{cl}^G(Q_j)\leq \mathcal{M}^G(P) + \operatorname{cl}^G([G, P]) = k + 1$, $\forall 1\leq j\leq 2$,
            \item Proposition \ref{proposition1},
            \item Lemma \ref{lemma4},
        \end{itemize}
        we conclude that $[G, P] \not= \{e\}$ and
        
    \begin{equation*}\begin{split}
        \operatorname{cl}^G(P) &\leq \operatorname{cl}^G(\underbrace{\gamma_1^G(P)}_{\mathclap{= [P, GP] = [Q_1Q_2, GP] \subseteq [Q_1, GP]\langle\!\langle [\langle\!\langle Q_2\rangle\!\rangle^{Q_1}, GP]\rangle\!\rangle^{Q_1} = [Q_1, GP][Q_2, GP]}}) + 1 \\ &\leq \operatorname{cl}^G(\underbrace{[Q_1, GP][Q_2, GP]}_{\mathclap{= [Q_1, G]\langle\!\langle [Q_1, \langle\!\langle P\rangle\!\rangle^G]\rangle\!\rangle^G [Q_2, G]\langle\!\langle [Q_2, \langle\!\langle P\rangle\!\rangle^G]\rangle\!\rangle^G = [Q_1, G][Q_1, P][Q_2, G][Q_2, P]}}) + 1 \end{split}\end{equation*}
        
        \begin{equation*}\begin{split} &= \operatorname{cl}^G(\underbrace{[Q_1, P][Q_2, P]}_{\mathclap{= P' := N\trianglelefteq GP}}\overbrace{[Q_1, G]}^{\mathclap{:= X_1\subseteq^G [G, P]}}\underbrace{[Q_2, G]}_{\mathclap{:= X_2\subseteq^G [G, P]}}) + 1 \\ &= \operatorname{cl}^G(NX_1X_2) + 1 \\ &\leq \operatorname{cl}^G(\underbrace{\gamma_k^G(NX_1X_2)}_{\mathclap{\subseteq \gamma_1^G([G, P])\gamma_k^G(NX_1)\gamma_k^G(NX_2)}}) + k + 1 \\ &\leq \operatorname{cl}^G(\underbrace{\gamma_1^G([G, P])}_{\mathclap{= \{e\}}}\gamma_k^G(\overbrace{NX_1}^{\mathclap{= P'[Q_1, G]}})\gamma_k^G(\underbrace{NX_2}_{\mathclap{= P'[Q_2, G]}})) + k + 1\\ &= \operatorname{cl}^G(\gamma_k^G(\overbrace{P'[Q_1, G]}^{\mathclap{= [Q_1, P][Q_1, G]\subseteq\gamma_1^G(Q_1)}})\gamma_k^G(\underbrace{P'[Q_2, G]}_{\mathclap{= [Q_2, P][Q_2, G]\subseteq\gamma_1^G(Q_2)}})) + k + 1\\ &\leq \operatorname{cl}^G(\underbrace{\gamma_{k+1}^G(Q_1)}_{\mathclap{\subseteq \gamma_{\operatorname{cl}^G(Q_1)}^G(Q_1) = \{e\}}}\overbrace{\gamma_{k+1}^G(Q_2)}^{\mathclap{\subseteq \gamma_{\operatorname{cl}^G(Q_1)}^G(Q_1) = \{e\}}}) + k + 1\\ &= \operatorname{cl}^G(\{e\}) + k + 1 = \mathcal{M}^G(P) + \operatorname{cl}^G([G, P]).\end{split}
    \end{equation*} $\hfill\Box$

\end{document}